\documentclass[copyright,creativecommons]{eptcs}

\providecommand{\event}{ACT 2025} %

\usepackage{iftex}

\ifpdf
  \usepackage{underscore}         %
  \usepackage[T1]{fontenc}        %
\else
  \usepackage{breakurl}           %
\fi

\usepackage[utf8]{inputenc}
\usepackage[english]{babel}
\usepackage[font=small,labelfont=bf]{caption}

\usepackage{caption}
\usepackage{paralist} %
\usepackage{csquotes}

\usepackage{amsmath}
\usepackage{amssymb}
\usepackage{amsthm}

\usepackage{mathtools}
\usepackage{IEEEtrantools}

\usepackage{eucal}
\usepackage{bbm}
\usepackage{stmaryrd} %
\usepackage{bbm} %

\usepackage{trimclip}

\usepackage{scalerel}

\usepackage[dvipsnames]{xcolor}
\definecolor{darkgreen}{RGB}{35, 89, 52}
\definecolor{paleorange}{RGB}{255, 236, 207}
\definecolor{paleyellow}{RGB}{255, 252, 217}
\definecolor{paleteal}{RGB}{223, 245, 243}
\definecolor{palered}{RGB}{255, 234, 232}
\definecolor{palepurple}{RGB}{238, 219, 255}
\definecolor{fungreen}{RGB}{0,137,40}
\definecolor{resred}{RGB}{198,42,8}

\usepackage{graphicx}
\graphicspath{ {./images/} }

\usepackage{hyperref}
\hypersetup{
	colorlinks=true,
	linkcolor=darkgreen,
	filecolor=magenta,      
	urlcolor=MidnightBlue,
	citecolor=darkgreen,
	pdftitle={} %
}

\usepackage{quiver} %

\usepackage{enumitem} %

\usepackage[capitalize, noabbrev]{cleveref} %

\newtheorem{counter}{Definition}[section]
\newtheorem{definition}[counter]{Definition}
\newtheorem{proposition}[counter]{Proposition}
\newtheorem{theorem}[counter]{Theorem}

\newtheorem{example}[counter]{Example}
\newtheorem{remark}[counter]{Remark}

\newcommand{\Reals}{\mathbb{R}}
\newcommand{\PosReals}{\Reals_{\geq 0}}

\newcommand{\then}{\fatsemi}

\newcommand{\Ob}{\mathsf{Ob}}

\newcommand{\cat}[1]{\mathcal{#1}}
\newcommand{\functor}[1]{\mathsf{#1}}
\newcommand{\twocat}[2]{{\mathbb{#1}\mathsf{#2}}}

\newcommand{\C}{\cat{C}}
\newcommand{\D}{\cat{D}}
\newcommand{\E}{\cat{E}}
\newcommand{\F}{\cat{F}}

\newcommand{\R}{\cat{R}}

\newcommand{\U}{\cat{U}}
\newcommand{\V}{\cat{V}}
\newcommand{\W}{\cat{W}}

\newcommand{\Kl}{\mathsf{Kl}} %
\newcommand{\Para}{\mathbf{Para}}

\newcommand{\Bool}{\mathsf{Bool}}
\newcommand{\Pos}{\mathsf{Pos}}

\newcommand{\Set}{\mathsf{Set}}
\newcommand{\FinSet}{\mathsf{FinSet}}
\newcommand{\Top}{\mathsf{Top}}

\newcommand{\smcat}{\mathsf{cat}} %
\newcommand{\Meas}{\mathsf{Meas}}
\newcommand{\DP}{\mathsf{DP}}
\newcommand{\PosDP}{\mathsf{DP}_\leq}

\newcommand{\Stoch}{\mathsf{Stoch}}

\newcommand{\Petri}{\mathsf{Petri}}
\newcommand{\Csp}{\mathbf{Csp}}

\newcommand{\copyMor}{\mathsf{cp}}
\newcommand{\delMor}{\mathsf{del}}

\newcommand{\SigAlg}{\Sigma}
\newcommand{\interval}[2]{[#1, #2]}

\newcommand{\prodSmcat}{\times}
\newcommand{\prodW}{\otimes}

\newcommand{\powerset}{\functor{Pow}}

\newcommand{\ParaU}[3]{(#1 / #2 #3)}

\newcommand{\PowNoEmpty}{\functor{Pow_\varnothing}}
\newcommand{\Dist}{\functor{D}} %
\newcommand{\Arr}{\functor{Arr}}
\newcommand{\TwiArr}{\functor{TwiArr}}
\newcommand{\leftFunctor}{\functor{L}}

\newcommand{\CCat}{\twocat{C}{at}} %
\newcommand{\MMoncat}{\twocat{M}{oncat}} %
\newcommand{\TTwocat}{\mathbbm{2}\mathsf{cat}} %
\newcommand{\VCat}{\twocat{V}{cat}}
\newcommand{\WCat}{\twocat{W}{cat}}

\newcommand{\CC}{\twocat{C}{}}

\makeatletter
\newcommand{\oset}[3][0ex]{%
  \mathrel{\mathop{#3}\limits^{
    \vbox to#1{\kern-2.2\ex@
    \hbox{$\scriptstyle#2$}\vss}}}}
\makeatother
\newcommand{\vertthen}{{\oset{\diamond}{\text{{\clipbox{0 0 0 3.4pt}{$\fatsemi$}}}}}}
\newcommand{\horthen}{{\oset{\star}{\text{{\clipbox{0 0 0 3.4pt}{$\fatsemi$}}}}}}

\newcommand{\op}[1]{#1^{\text{op}}}
\newcommand{\id}{\mathsf{id}}
\newcommand{\Id}{\mathsf{Id}}

\newcommand{\monadM}{\functor{M}}
\newcommand{\monadUnit}{\eta}
\newcommand{\monadMul}{\mu}
\newcommand{\monadTup}{(\monadM, \monadMul, \monadUnit)}
\newcommand{\symMonadMor}{\nabla}
\newcommand{\symMonadMorOf}[2]{\symMonadMor_{#1,#2}}

\newcommand{\commacat}[2]{(#1 \downarrow #2)} %

\newcommand{\sliceWith}[1]{/#1}
\newcommand{\sliceFunctor}{\sliceWith{(-)}}

\newcommand{\unitObj}{1}
\newcommand{\unitOf}[1]{\unitObj_#1}
\newcommand{\unitW}{I_\W}
\newcommand{\idUnitW}{\id_{\unitW}}
\newcommand{\unitCat}{\unitOf{{\smcat}}}

\newcommand{\poset}[1]{\cat{#1}}
\newcommand{\posetP}{\poset{P}}
\newcommand{\posetQ}{\poset{Q}}
\newcommand{\posetR}{\poset{R}}
\newcommand{\posetF}{\poset{F}}
\newcommand{\posetleq}{\preceq}

\newcommand{\designProblem}{d}

\newcommand{\elementTo}{\mapsto}
\newcommand{\functorArr}{\to}

\newcommand{\morphism}[1]{#1}
\newcommand{\ChassisDP}{\morphism{C}}
\newcommand{\BatteryDP}{\morphism{B}}
\newcommand{\monfunPhi}{\varphi}
\newcommand{\paramTheta}{\theta}

\newcommand{\True}{\top}
\newcommand{\False}{\bot}

\newcommand{\funsty}[1]{\textcolor{fungreen}{#1}}
\newcommand{\ressty}[1]{\textcolor{resred}{#1}}

\newcommand{\FixFunMinRes}{\emph{Fix\funsty{Fun}Min\ressty{Res}}}

\title{Composable Uncertainty in Symmetric Monoidal Categories for Design Problems}

\def\titlerunning{Composable Uncertainty in Symmetric Monoidal Categories for Design Problems}

\def\authorrunning{M. Furter, Y. Huang, G. Zardini}
\author{
Marius Furter\thanks{Lead author}
\institute{University of Zurich \\ Zurich, Switzerland}
\email{marius.furter@math.uzh.ch}
\and
Yujun Huang
\institute{Massachusetts Institute of Technology \\ Cambridge, MA, USA}
\email{yujun233@mit.edu}
\and
Gioele Zardini
\institute{Massachusetts Institute of Technology \\ Cambridge, MA, USA}
\email{gzardini@mit.edu}
}

\begin{document}
\maketitle

\begin{abstract}
Applied category theory often studies symmetric monoidal categories (SMCs) whose morphisms represent open systems.
These structures naturally accommodate complex wiring patterns, leveraging (co)monoidal structures for splitting and merging wires, or compact closed structures for feedback. 
A key example is the compact closed SMC of design problems ($\DP$), which enables a compositional approach to co-design in engineering.
However, in practice, the systems of interest may not be fully known.
Recently, Markov categories have emerged as a powerful framework for modeling uncertain processes. 
In this work, we demonstrate how to integrate this perspective into the study of open systems while preserving consistency with the underlying SMC structure.
To this end, we employ the change-of-base construction for enriched categories, replacing the morphisms of a symmetric monoidal $\V$-category $\C$ with parametric maps $A \to \C(X,Y)$ in a Markov category induced by a symmetric monoidal monad. 
This results in a symmetric monoidal 2-category $N_*\C$ with the same objects as $\C$ and reparametrization 2-cells. 
By choosing different monads, we capture various types of uncertainty. 
The category underlying $\C$ embeds into $N_*\C$ via a strict symmetric monoidal functor, allowing (co)monoidal and compact closed structures to be transferred. 
Applied to $\DP$, this construction leads to categories of practical relevance, such as parametrized design problems for optimization, and parametrized distributions of design problems for decision theory and Bayesian learning.
\end{abstract}

\maketitle

\section{Introduction}
The design of complex systems is a critical but challenging task for society.
Difficulties arise from interactions between heterogeneous components, competing design objectives, diverse stakeholders with varied interests, and computationally intensive models.
Co-design tackles these challenges by leveraging a categorical model of design problems called $\DP$ \cite{zardiniCoDesignComplexSystems2023}.
$\DP$ is a compact closed symmetric monoidal category (SMC), whose objects are partially ordered sets (posets), and whose morphisms are monotone maps from $\op{\posetF} \times \posetR$ to the poset of truth values $\Bool = \{ \False \leq \True \}$.
It models heterogeneous systems as processes providing \emph{functionalities} $f \in \posetF$ in return for \emph{resources} $r \in \posetR$.
Monotonicity captures the intuition that, if a resource $r$ suffices to provide functionality $f$, then it also suffices for any worse functionality $f' \leq f$. Moreover, any better resource $r' \geq r$ should also suffice to provide $f$.
Posets present transparent trade-offs between design objectives and provide a common ground for stakeholder discussions.
Moreover, $\DP$ enables functorial \emph{queries} which decompose complex design problems into simpler ones that can be solved efficiently.

Current co-design methods address uncertainty using upper and lower bounds \cite{censiUncertainty2017}.
This approach is robust and well-suited for safety-critical applications, ensuring system performance even in worst-case scenarios.
Furthermore, it is effective for approximating and solving continuous co-design via discretization.
However, this method is unable to provide quantitative measures of uncertainty, such as the probability that a system will meet functional requirements given specific resource constraints.
Understanding such trade-off between functionality, resource usage, and success probability is crucial in many cases.

Furthermore, design problems often involve \emph{parametric} uncertainties linked to design choices.
For instance, when designing soft-robot manipulators, the choice of material introduces the elastic modulus as a parameter that follows a distribution.
Moreover, its optimal value differs between components: a lower elastic modulus benefits hardware design by lowering driving force requirements, while a higher modulus benefits controller design by improving resistance to disturbances.
Such examples highlight the need to incorporate quantitative uncertainty into co-design in a way that depends on design parameters.

We propose a general scheme for adding parametrized uncertainty to any SMC.
Importantly, this process preserves compact closed and (co)monoidal structure.
Our approach supports any uncertainty semantics modeled by symmetric monoidal monads. 
This includes intervals, subsets and distributions, whose Kleisli categories form Markov categories.
The parametrization enables decision-making and learning for any process represented by a diagram in the SMC.

In more detail, given a $\V$-category $\C$, we replace its arrows $\C(X,Y)$ with maps $A \to \monadM \C(X,Y)$, where $\monadM \colon \V \to \V$ is a symmetric monoidal monad. The strength $\symMonadMor \colon \monadM(-) \otimes \monadM(=) \to \monadM(- \otimes =)$ allows us to compose $f \colon A \to \monadM \C(X,Y)$ and $g \colon B \to \monadM \C(Y,Z)$ according to
\[
A \otimes B \overset{f \otimes g}{\longrightarrow}  \monadM \C(X,Y) \otimes \monadM \C(Y,Z) \overset{\symMonadMor}{\longrightarrow} \monadM( \C(X,Y) \otimes \C(Y,Z) ) \overset{\monadM \then_\C}{\longrightarrow} \monadM \C(X,Z).
\]
In addition, any $\V$-arrow $\varphi \colon A' \to \monadM A$ allows us to reparametrize $f$ to $A' \to \monadM \C(X,Y)$ using Kleisli composition, providing 2-categorical structure. 
Furthermore, when $\C$ is monoidal, we can define a monoidal product analogous to composition. Hence, our construction promises to be a monoidal 2-category.

We prove this using change-of-base: Given a monoidal functor $N \colon \V \to \W$, we can convert a $\V$-category $\C$ into a $\W$-category $N_*\C$ by replacing its arrows by $N\C(X,Y)$. 
Moreover, when $N$ is symmetric and $\C$ (symmetric) monoidal, then so is $N_* C$. 
Finally, there is an inclusion of $\C_0 \to (N_* \C)_0$ of underlying categories that strictly preserves the monoidal product and coherence maps.

We view the proposed construction as a change-of-base along $\leftFunctor_\monadM \colon \V \to\Kl_\monadM $ and $\Kl_\monadM \sliceFunctor \colon \Kl_\monadM \to \smcat$, where $\leftFunctor_\monadM$ is the identity-on-objects functor sending $f \mapsto f \then \monadUnit$, and $\Kl_\monadM \sliceFunctor$ maps $X$ to the slice category $\Kl_\monadM \sliceWith{X}$. 
However, $\Kl_\monadM \sliceWith{X}$ is only symmetric monoidal up to invertible 2-cells. Nonetheless, when $\Kl_\monadM$ is strict (as can be assumed by strictification), change-of-base yields a (symmetric) monoidal 2-category where interchange (and naturality of symmetry) hold up to coherent reparametrization isomorphisms.

Applied to design problems, this construction leads to categories of practical relevance. 
We demonstrate that $\DP$ may be viewed as enriched in sets, posets, and measurable spaces. 
Hence, we can choose monads in any of these setting to get parametrized subsets, intervals, and distributions of design problems. 
In all cases, the compact closed SMC structure transfers to the resulting 2-category.

The remainder of the paper is structured as follows. \Cref{sec:background} introduces the required background on enriched categories and change-of-base.
\Cref{sec:parametric-uncertainty} describes uncertainty monads and the partial slice functor. 
It concludes with an explicit description of the proposed scheme in \cref{sec:main-construction}.
Finally, \cref{sec:learning-dps} describes categories of uncertain DPs and demonstrates their usefulness. 
An extended version of this paper including full definitions and detailed proofs is available at \cite{furterComposable2025Extended}.
\paragraph{Acknowledgments.} MF thanks the Digital Society Initiative for funding his Ph.D., the Rudge (1948) and Nancy Allen Chair for his visit to LIDS at MIT, and SNF Grant 200021\_227719 for travel costs.

\section{Background}\label{sec:background}
This section introduces enriched categories and the change-of-base construction. 
Moreover, it explains that change-of-base is 2-functorial and induces a transfer of symmetric monoidal structure. 
The latter result originates from Geoff Cruttwell's Ph.D.~thesis~\cite{cruttwellNormed2008}, which provides an exceptionally clear exposition.
We assume familiarity with SMCs \cite{macLaneCategories1998}, which we usually denote $(\V, \otimes, I)$ with coherence maps $\alpha, \rho, \lambda$, and $\sigma$. Many SMCs we use are Cartesian, including $\Set$ (sets and functions), $\Pos$ (posets and monotone functions), $\Top$ (topological spaces and continuous functions), $\Meas$ (measurable spaces and measurable functions), and $\smcat$ (small categories and functors). Our 2-categories will always be \emph{strict} and we use $f \then g := g \circ f$ for composition in diagrammatic order.

\begin{example}[Design problems]
    Let $\Bool := \{ \bot \leq \top\}$ be the poset of truth values. Given posets $\F$ and $\R$, a \emph{feasibility relation} $\Phi \colon \F \to \R$ is a monotone map $\op{\F} \times \R \to \Bool$. Posets and feasibility relations assemble into a compact closed SMC $\DP$ of \emph{design problems} \cite{fongInvitation2019, zardiniCoDesignComplexSystems2023} under 
    \[
        (\Phi \then \Psi)(f,q) := \bigvee_{r \in \R} \Phi(f,r) \wedge \Psi(r,q),
        \qquad \quad
        (\Phi_1 \otimes \Phi_2)[(f_1,f_2),(r_1,r_2)] := \Phi_1(f_1,r_1) \wedge \Phi_2(f_2,r_1),
    \]
    where we define $\posetF_1 \otimes \posetF_2 := \posetF_1 \times \posetF_2$ to be the Cartesian product of posets.
\end{example}

We recall monoidal functors due to their key role in the change-of-base construction.

\begin{definition}[Monoidal functor]
    Let $(\V,\otimes,I)$ and $(\W,\bullet, J)$ be monoidal categories. A \emph{(lax) monoidal functor} $N \colon \V \to \W$ is a functor from $\V$ to $\W$, along with natural transformations whose components 
    \[
        N_\epsilon \colon J \to NI, \qquad \qquad
        \widetilde N_{A,B} \colon NA \bullet NB \to N(A \otimes B)
    \]
    satisfy associativity and unitality conditions.
    If $N_\epsilon$ and $\widetilde N$ are isomorphisms, we call $N$ \emph{strong}; if they are identities, we call $N$ \emph{strict}. If $\V$ and $\W$ are symmetric, then $N$ is called \emph{symmetric} if it preserves $\sigma$ in the sense of $\widetilde N_{A,B} \then N(\sigma_\V) = \sigma_\W \then \widetilde N_{B,A}$.
\end{definition}
Just as small categories, functors, and natural transformations assemble into a 2-category $\CCat$, their monoidal counterparts also form a 2-category called $\MMoncat$.

\subsection{Enriched categories}
Enriched categories \cite{kellyBasic1982} replace hom-sets by more general hom-objects that inhabit a monoidal category $\V$. 
We will always assume $\V$ to be symmetric so that the monoidal product of $\V$-categories is defined.

\begin{definition}[$\V$-category]
    Given an SMC $(\V, \otimes, I)$, a \emph{$\V$-category} $\C$ consists of
    \begin{compactitem}
        \item[(i)] a set of \emph{objects} $\Ob(\C)$, and for all $A,B,C \in \Ob(\C)$,
        \item[(ii)] a $\V$-object of \emph{arrows} $\C(A,B) \in \Ob(\V)$,
        \item[(iii)] \emph{composition} arrows $\then_{A,B,C} \colon \C(A,B) \otimes \C(B,C) \to \C(A,C),$
        \item[(iv)] and \emph{identity} arrows $\id_A \colon I \to \C(A,A)$,
    \end{compactitem}
    that satisfy the following unitality and associativity conditions:
    \[
    \begin{tikzcd}[cramped]
    	{(\C(A,B) \otimes \C(B,C)) \otimes \C(C,D)} & {\C(A,C) \otimes \C(C,D)} \\
    	{\C(A,B) \otimes (\C(B,C) \otimes \C(C,D))} \\
    	{\C(A,B) \otimes \C(B,D)} & {\C(A,D)}
    	\arrow["{\then \otimes 1}", from=1-1, to=1-2]
    	\arrow["\alpha"', from=1-1, to=2-1]
    	\arrow["\then", from=1-2, to=3-2]
    	\arrow["{1 \otimes \then}"', from=2-1, to=3-1]
    	\arrow["\then"', from=3-1, to=3-2]
    \end{tikzcd}
    \qquad\qquad
    \begin{tikzcd}[cramped, sep = scriptsize]
    	{I \otimes \C(A,B)} \\
    	{\C(A,A) \otimes \C(A,B)} & {\C(A,B)} \\
    	{\C(A,B) \otimes I} \\
    	{\C(A,B) \otimes \C(B,B)} & {\C(A,B)}
    	\arrow["{\id_A \otimes 1}"', from=1-1, to=2-1]
    	\arrow["\lambda", from=1-1, to=2-2]
    	\arrow["\then"', from=2-1, to=2-2]
    	\arrow["{1 \otimes \id_B}"', from=3-1, to=4-1]
    	\arrow["\rho", from=3-1, to=4-2]
    	\arrow["\then"', from=4-1, to=4-2]
    \end{tikzcd}
\]
\end{definition}

\begin{definition}[$\V$-functor]
    Let $\C$ and $\D$ be $\V$-categories. A \emph{$\V$-functor} $F \colon \C \to \D$ consists of a function $F_\Ob \colon \Ob(\C) \to \Ob(\D)$, along with $\V$-arrows $F_{A,B} \colon \C(A,B) \to \D(F_\Ob (A),F_\Ob(B) )$, indexed by objects $A,B \in \Ob(\C)$. These data must preserve composition and identities:
    \[\begin{tikzcd}[cramped, row sep = scriptsize]
    	{\C(A,B) \otimes \C(B,C)} && {\D(FA,FB) \otimes \D(FB,FC)} \\
    	{\C(A,C)} && {\D(FA,FC)}
    	\arrow["{F_{A,B} \otimes F_{B,C}}", from=1-1, to=1-3]
    	\arrow["\then"', from=1-1, to=2-1]
    	\arrow["\then", from=1-3, to=2-3]
    	\arrow["{F_{A,C}}", from=2-1, to=2-3]
    \end{tikzcd}
    \qquad \qquad
    \begin{tikzcd}[cramped, row sep = scriptsize]
    	I \\
    	{\C(A,A)} & {\D(FA,FA)}
    	\arrow["{\id_A}"', from=1-1, to=2-1]
    	\arrow["{\id_{FA}}", from=1-1, to=2-2]
    	\arrow["{F_{A,A}}", from=2-1, to=2-2]
    \end{tikzcd}\]
    Given $\V$-functors $F \colon \C \to \D$ and $G \colon \D \to \E$, we define their \emph{composite} $F \then G \colon \C \to \E$ on objects as $(F \then G)_\Ob := F_\Ob \then G_\Ob$, and on arrows as $(F \then G)_{A,B} := F_{A,B} \then G_{FA,FB}$. 
\end{definition}

\begin{definition}[$\V$-natural transformation]
    Let $F,G : \C \to \D$ be $\V$-functors. A \emph{$\V$-natural transformation} $\tau: F \Rightarrow G$ is a family of $\V$-arrows $\tau_A: I \to \D(FA,GA)$ that satisfy a $\V$-naturality condition.
    They compose both vertically and horizontally.
\end{definition}
As is the case for ordinary categories, $\V$-categories, $\V$-functors, and $\V$-natural transformations assemble into a 2-category $\VCat$ \cite[Thm.~10.2]{eilenbergClosed1966}.

\begin{example}
     $\Set$-categories are locally small categories. Similarly, $\Set$-functors and $\Set$-natural transformations can be identified with their ordinary counterparts.
\end{example}

\begin{example}
    A $\smcat$-enriched category $\twocat{C}{}$ is a locally small 2-category. For any objects $A,B$, we have a hom-category $\C(A,B)$. We call the objects of $\twocat{C}{}$ \emph{0-cells}, the objects of $\C(A,B)$ \emph{1-cells}, and the arrows of $\C(A,B)$ \emph{2-cells}. There are two directions of composition: Composition in $\C(A,B)$, denoted by $\vertthen$, will be called \emph{vertical}, while the composition functor, denoted $\horthen \colon \C(A,B) \times \C(B,C) \to \C(A,C)$, will be called \emph{horizontal composition}. For more details, see \cite[Chapter 2.3]{johnson2Dimensional2020}.
\end{example}

\begin{example}
    $\smcat$-functors are 2-functors. Explicitly, a 2-functor $F \colon \twocat{C}{} \to \twocat{D}{}$ maps 0-cells with $F_\Ob$, and 1- and 2-cells with functors $F_{A,B} \colon \twocat{C}{}(A,B) \to \twocat{D}{}(FA,FB)$ in a way that strictly preserves horizontal composition and identities. Hence, $F$ preserves both vertical and horizontal compositional structures.
\end{example}

\begin{example}\label{ex:DP-pos-enriched}
    We can view $\DP$ as enriched in $\Pos$: For posets $\F,\R$, the set of feasibility relations $\DP(\F,\R)$ has a natural point-wise ordering. Since $\wedge$ and $\vee$ are monotone operations, composition of feasibility relations is monotone with respect to this ordering. Associativity and unitality mean the same in $\Pos$ as they do in $\Set$. $\Pos$-functors are ordinary functors whose on-arrows maps are monotone.
\end{example}

\begin{definition}[Category underlying a $\V$-category]
    We write $(-)_0 \colon \VCat \to \CCat$ for the representable 2-functor $\VCat(I_\VCat,-)$, where $I_\VCat$ is the $\V$-category with a single object $*$ and arrows $I_\VCat(*,*) := I$. This 2-functor sends a $\V$-category $\C$ to the ordinary category $\C_0$ whose objects are those of $\C$, and whose arrows $f \colon A \to B$ are $\V$-arrows of the form $f \colon I \to \C(A,B)$. Composition is defined by
    \[ I \cong I \otimes I \overset{f \otimes g}{\longrightarrow} \C(A,B) \otimes \C(B,C) \overset{\then}{\longrightarrow} \C(A,C). \]
    A $\V$-functor $F \colon \C \to \D$ is mapped to the functor $F_0 \colon \C_0 \to \D_0$ that sends $A \mapsto FA$ and $f \colon I \to \C(A,B)$ to $Ff := f \then F_{A,B}$. Finally, a $\V$-natural transformation $\tau \colon F \Rightarrow G$ with components $\tau_A \colon I \to \D(FA,GA)$ is sent to the natural transformation $\tau_0 \colon F_0 \Rightarrow G_0$ with components $(\tau_0)_A := \tau_A$.
\end{definition}

\begin{definition}[Monoidal product of $\V$-categories.] \label{def:monoidal-prod-Vcats}
    Given $\V$-categories $\C$ and $\D$, their \emph{monoidal product} $\C \otimes \D$ has objects $\Ob(\C \otimes \D) := \Ob(\C) \times \Ob(\D)$ and arrows $(\C \otimes \D)[(A_1,A_2),(B_1,B_2)] := \C(A_1,B_1) \otimes \D(A_2,B_2)$.  
    Composition and identities are given by
    \[
    \begin{tikzcd}[cramped, sep = scriptsize]
    	{[\C(A_1,B_1) \otimes \D(A_2, B_2)] \otimes [\C(B_1,C_1) \otimes \D(B_2, C_2)]} \\
    	{[\C(A_1,B_1) \otimes \C(B_1, C_1)] \otimes [\D(A_2,B_2) \otimes \D(B_2, C_2)]} & {\C(A_1,C_1) \otimes \D(A_2,C_2)}
    	\arrow["m"', from=1-1, to=2-1]
    	\arrow["{\then_{\C \otimes \D}}", dashed, from=1-1, to=2-2]
    	\arrow["{\then_\C \otimes \then_\D}"', from=2-1, to=2-2]
    \end{tikzcd}
    \qquad
    \begin{tikzcd}[cramped, sep = scriptsize]
    	{I \cong I \otimes I } \\
    	{\C(A,A) \otimes \D(B,B)}
    	\arrow["{\id_A \otimes \id_B}"', from=1-1, to=2-1]
    \end{tikzcd}
    \]
    where $m$ is the unique coherence isomorphism composed of $\alpha$ and $\sigma$. This product has $I_\VCat$ as unit.
\end{definition}

\begin{definition}
    Let $\nu \colon \C_0 \times \C_0 \to (\C \otimes \C)_0$ be the identity-on-objects functor sending pairs $f_i \colon I \to \C(A_i,B_i)$ to $f_1 \otimes f_2$.
    Given a $\V$-functor $F \colon \C \otimes \C \to \C$, we define $\bar F_0$ as $\nu \then F_0 \colon \C_0 \times \C_0  \to \C_0$.
\end{definition}

\begin{definition}[Monoidal $\V$-category]
    A \emph{monoidal $\V$-category} $(\C, \boxtimes, J)$ consists of a $\V$-category $\C$, a $\V$-functor $\boxtimes \colon \C \otimes \C \to \C$, a unit object $J \in \Ob(\C)$, and $\V$-natural isomorphisms with components
    \[ 
        a_{A,B,C} \colon I \to \C((A \boxtimes B) \boxtimes C, A \boxtimes (B \boxtimes C)), \qquad 
        r_A \colon I \to \C(A \boxtimes J, A), \qquad
        l_A \colon I \to \C(J \boxtimes A, A),
    \]
    that make $(\C_0, \bar \boxtimes_0, J)$ into a monoidal category.
    We call $(\C,\boxtimes, J)$ \emph{symmetric}, if it comes equipped with a $\V$-natural isomorphism $s_{A,B} \colon I \to \C(A \boxtimes B, B \boxtimes A)$ that makes $(\C_0,\bar \boxtimes_0, J)$ symmetric monoidal.
\end{definition}

\begin{example}
    A $\Set$-enriched SMC may be identified with an ordinary SMC. A $\Pos$-enriched SMC is an ordinary SMC whose hom-sets are posets, and composition and the monoidal product are monotone.
\end{example}

\begin{example}\label{ex:DP-pos-enriched-smc}
    The monoidal product in $\DP$ is monotone (since $\wedge$ is). Hence, we may view $\DP$ as a symmetric monoidal $\Pos$-category.
\end{example}

\begin{example} \label{ex:monoidal-cat-category}
    A monoidal $\smcat$-category is a strict instance of a monoidal 2-category, as defined in \cite{stayCompact2016}. It consists of a 2-category $\CC$ with a 2-functor $\boxtimes \colon \CC \times \CC \to \CC$, unit object $J \in \Ob(\CC)$, and invertible 1-cells $a_{A,B,C} \colon (A \boxtimes C) \boxtimes D \cong  A \boxtimes (C \boxtimes D)$, $r_A \colon A \boxtimes J \cong A$, and $l_A \colon J \boxtimes A \cong A$ that are natural in $A,B,C$, and satisfy the pentagon and triangle identities. A symmetric structure on $\CC$ consists of 1-cells $s_{A,B} \colon A \boxtimes B \to B \boxtimes A$ that are natural in $A,B$, satisfy $s_{A,B} \then s_{B,A} = \id$ and the coherence axioms for SMCs.
\end{example}

Our main construction in \cref{thm:main-construction} nearly produces symmetric monoidal $\smcat$-categories.

\begin{definition} \label{def:nearly-symmetric-monoidal-2cat}
    A \emph{nearly strict monoidal 2-category} is a monoidal $\smcat$-category where $\boxtimes \colon \CC \times \CC \to \CC$ is only required to be a pseudofunctor  \cite[Def.~4.1.2]{johnson2Dimensional2020}. 
    In particular, interchange only holds up to coherent invertible 2-cells $\vartheta_{f_1,f_2,g_1,g_2} \colon (f_1 \boxtimes f_2) \then (g_1 \boxtimes g_2) \Rightarrow (f_1 \then g_1) \boxtimes (f_2 \then g_2)$
    called \emph{tensorators}.
    A \emph{nearly strict symmetric monoidal 2-category} is a nearly strict monoidal 2-category together with a strong 2-natural transformation \cite[Def.~4.2.1]{johnson2Dimensional2020} with 1-cell components $s_{A,B} \colon A \boxtimes B \to B \boxtimes A$ that satisfy $s_{A,B} \then s_{B,A} = \id$ and the coherence axioms for SMCs. Thus, symmetry coherence is only required to be natural up to invertible 2-cells.
    Nearly strict (symmetric) monoidal 2-categories are still particularly strict instances of the (symmetric) monoidal 2-categories of \cite{stayCompact2016}.
\end{definition}

\subsection{Change of base}

Given a monoidal functor $N \colon \V \to \W$, we can map the hom-objects of a $\V$-category $\C$ along $F$ to get a $\W$-category. This was first shown in \cite{eilenbergClosed1966}, with clear proofs available in \cite[Props.~4.2.1-3, Thm.~4.2.4]{cruttwellNormed2008}.

\begin{proposition}[Change-of-base]\label{prop:change-of-base}
    Any monoidal functor $N \colon \V \to \W$ induces a 2-functor $N_* \colon \VCat \to \WCat$, given by change-of-base along $N$. Explicitly, $N_*$ is defined as follows:
    \begin{compactitem}
        \item[(i)] A $\V$-category $\C$ is sent to the $\W$-category $F_*\C$ which has the same objects as $\C$, arrows given by $(N_*\C)(A,B) := N\C(A,B)$, with composition and identities
        \begin{align*}
            \then_{N_*\C} &:= N\C(A,B) \otimes_W N\C(B,C) \overset{\widetilde N}{\longrightarrow} N( \C(A,B) \otimes_\V \C(B,C) ) \overset{N\then_\C}{\longrightarrow} N\C(A,C), \\
            \id_{*A} &:= I_\W \overset{N_\epsilon}{\longrightarrow} NI_\V \overset{N\id_A}{\longrightarrow} N\C(A,A).
        \end{align*}        
    \item[(ii)] A $\V$-functor $F \colon \C \to \D$ maps to the $\W$-functor $N_*F \colon N_*\C \to N_*\D$ which acts on objects by $(N_*F)_\Ob := F_\Ob$ and on arrows by $(N_*F)_{A,B} := NF_{A,B}$.
    \item[(iii)] A $\V$-natural transformation $\tau \colon F \Rightarrow G$ maps to the $\W$-natural transformation $N_*\tau \colon N_*F \Rightarrow N_*G$ with components $(N_* \tau)_A := I_\W \overset{N_\epsilon}{\longrightarrow} NI_\V \overset{N\tau_A}{\longrightarrow} N\D(FA,GA)$.
    \end{compactitem}
\end{proposition}

Change-of-base respects monoidal transformations. The following is \cite[Prop.~4.3.1, Thm.~4.3.2]{cruttwellNormed2008}.

\begin{proposition} \label{prop:change-of-base-2-functorial}
    The change-of-base construction induces a 2-functor  $(-)_* \colon \MMoncat \to \TTwocat$. It sends a monoidal category $\V$ to the 2-category $\VCat$, a monoidal functor $N \colon \V \to \W$ to the 2-functor $N_* \colon \VCat \to \WCat$, and a monoidal transformation $\tau \colon N \Rightarrow M$ to the 2-natural transformation $\tau_* \colon N_* \Rightarrow M_*$ whose component $\tau_{*\C} \colon N_*\C \to M_*\C$ is the $\W$-functor with $(\tau_{*\C})_\Ob = \id_{\Ob(\C)}$ and $(\tau_{*\C})_{A,B} := \tau_{\C(A,B)}$.
\end{proposition}
Finally, symmetric monoidal structure transfers along the change-of-base, provided $N$ is symmetric.

\begin{proposition} \label{prop:change-of-base-monoidal-transfer}
    Let $N \colon \V \to \W$ be a symmetric monoidal functor and $(\C, \boxtimes, J)$ a (symmetric) monoidal $\V$-category. Then $N_*\C$ is a (symmetric) monoidal $\W$-category with $\W$-functor $\boxtimes^N \colon N_*\C \otimes N_*\C \to N_*\C$ which maps objects by $(A,B) \mapsto A \boxtimes B$ and arrows according to 
    \[ 
        N\C(A_1,B_1) \otimes N\C(A_2,B_2) \overset{\widetilde N}{\longrightarrow} N(\C(A_1,B_1) \otimes \C(A_2,B_2)) \overset{N\boxtimes}{\longrightarrow} N\C(A_1 \boxtimes A_2, B_1 \boxtimes B_2).
    \]
    The unit is given by $J \in \Ob(\C) = \Ob(N_*\C)$ and the coherence isomorphisms have components $N_*a$, $N_*r$, $N_*l$ (and $N_*s$). Furthermore, $N$ induces a strict (symmetric) monoidal functor $N_*^0 \colon \C_0 \to (N_*\C)_0$ between the underlying categories that is identity-on-objects and sends $f \colon I_V \to \C(A,B)$ to $N_*f = N_\epsilon \then Nf$.
\end{proposition}
\begin{proof}
    Cruttwell provides a conceptual proof in \cite[Thm.~5.7.1]{cruttwellNormed2008} for transferring monoidal structure using that monoidal $\V$-categories are pseudomonoids in $\VCat$ and thus are preserved by the monoidal $(-)_*$. We give a direct explanation. When $N$ is symmetric, $\widetilde N \colon N(-) \bullet N(=) \to N(- \otimes =)$ is a monoidal transformation \cite[Prop.~5.3.6]{cruttwellNormed2008} and hence induces $\W$-functor $(\widetilde N)_* \colon N_*\C \otimes N_*\C \to N_*(\C \otimes \C)$. Thus, $\boxtimes^N$ can be obtained as the composition of $\W$-functors $(\widetilde N)_* \then (N_*\boxtimes)$.
    Using the fact that $N$ is symmetric monoidal, one can show that the proposed coherence isomorphisms have the correct source and target functors.
    To see that they satisfy the SMC axioms, one observes that $N_*^0$ is a strict monoidal functor: For $f \in \C_0(A,B)$ and $g \in \C_0(B,C)$ one has $N_*(f \then g) = (N_* f) \then (N_* g)$. Similarly, for $f_i \in \C_0(A_i, B_i)$ one has $N_*(f_1 \bar \boxtimes_0 f_2) = (N_*f_1) \bar \boxtimes^N_0 (N_*f_1)$. Moreover, when $\C$ is symmetric, so is $N_*^0$, since $N_*^0s = N_*s$. This uses the definitions, the fact that $N$ is monoidal, and properties of SMCs,
    \[
    \begin{tikzcd}[row sep = small]
    	{I_\W} & {I_\W \otimes I_\W} & {NI_\V \otimes NI_\V} & {N\C(A_1,B_1) \otimes N\C(A_2,B_2)} \\
    	& {I_\W \otimes NI_\V} \\
    	{N(I_\V)} && {N(I_\V \otimes I_\V)} & {N(\C(A_1,B_1) \otimes \C(A_2,B_2))} \\
    	&&& {N(\C(A_1\boxtimes A_2, B_1 \boxtimes B_2))}
    	\arrow["{\lambda^{-1}}", from=1-1, to=1-2]
    	\arrow["{N_\epsilon}"', from=1-1, to=3-1]
    	\arrow["{N_\epsilon \otimes N_\epsilon}", from=1-2, to=1-3]
    	\arrow[from=1-2, to=2-2]
    	\arrow["{Nf_1 \otimes Nf_2}", from=1-3, to=1-4]
    	\arrow["{\widetilde N}", from=1-3, to=3-3]
    	\arrow["{(\text{nat } \widetilde N)}"{description}, draw=none, from=1-3, to=3-4]
    	\arrow["{\widetilde N}", from=1-4, to=3-4]
    	\arrow["{(\text{nat } \lambda)}"{description}, draw=none, from=2-2, to=1-1]
    	\arrow[""{name=0, anchor=center, inner sep=0}, from=2-2, to=1-3]
    	\arrow["\lambda"{description}, from=2-2, to=3-1]
    	\arrow["{(N \text{ monoidal})}"{description}, draw=none, from=2-2, to=3-3]
    	\arrow["{N(\rho^{-1})}"', from=3-1, to=3-3]
    	\arrow["{N(f_1 \otimes f_2)}"', from=3-3, to=3-4]
    	\arrow["{N(*)}", from=3-4, to=4-4]
    	\arrow[draw=none, from=1-2, to=0]
    \end{tikzcd}
    \]
    where $(*)$ denotes either $\then$ or $\boxtimes$. Remarkably, it does not require $\W$-functoriality of $\boxtimes^N$. 
    It follows that the coherence diagrams of $\C_0$ also hold for $N_*a$, $N_*r$, $N_*l$ and $N_*s$ in $(N_*\C)_0$ when mapped via $N_*^0$.
\end{proof}

\section{Uncertainty with external parametrization}\label{sec:parametric-uncertainty}

This section applies the change-of-base techniques of \cref{sec:background} to equip SMCs with parametric uncertainty. 
We start by describing how to represent uncertainty semantics using symmetric monoidal monads.
Next, we introduce the partial slice construction used to obtain external parametrization. 
Finally, we present our main construction which replaces the arrows of a symmetric monoidal $\V$-category $\C$ by parametric maps $A \to \C(X,Y)$ in the Markov category arising from an uncertainty monad.
We further explain how structures of the original SMC can be transferred to the resulting 2-category.

\subsection{Uncertainty monads}\label{subsec:uncertainty-monads}

Many uncertainty semantics can be viewed as generalized collections of deterministic objects. 
For instance, subsets, intervals, and distributions of objects follow this pattern. Valuing maps in such generalized collections provides a model for uncertain processes.
This idea is formalized by symmetric monoidal monads \cite[Section 3]{fritzSynthetic2020} whose Kleisli categories often form Markov categories. 
The Markov category axioms \cite[Section 2]{fritzSynthetic2020} capture how such uncertain processes compose.

\begin{definition}[Symmetric monoidal monad]
    Let $\V$ be an SMC. A monad $\monadTup$ on $\V$ with multiplication $\monadMul$ and unit $\monadUnit$
    is called \emph{symmetric monoidal} if it comes equipped with morphisms
    $$
    \symMonadMorOf{X}{Y} \colon \monadM(X) \otimes \monadM(Y) \to \monadM(X \otimes Y)
    $$
   that are natural in $X$ and $Y$, make $\monadM$ into a symmetric monoidal functor with unit $\monadUnit_I$, and make $\monadMul$ and $\monadUnit$ into monoidal transformations.
    If $I$ is terminal in $\V$, and $\monadM I \cong I$, we call $\monadM$ \emph{affine}.
\end{definition}

Kleisli categories of symmetric monoidal monads are SMCs. The following result is proved in \cite[Prop.~3.1, Cor.~3.2]{fritzSynthetic2020}.

\begin{proposition} \label{prop:monad-markov-category}
    Given a symmetric monoidal monad $(\monadM, \monadUnit, \monadMul, \symMonadMor)$ on $\V$, its Kleisli category $\Kl_\monadM$ is symmetric monoidal under the product that sends $f \colon A \to \monadM X$ and $g \colon B \to \monadM Y$ to 
    \[
    A \otimes B \overset{f \otimes g}{\longrightarrow} \monadM X \otimes \monadM Y \overset{\symMonadMor}{\longrightarrow} \monadM(X \otimes Y).
    \]
    Moreover, the identity-on-objects functor $\leftFunctor_\monadM \colon \V \to \Kl_\monadM$ sending $f \mapsto f \then \monadUnit$ is strict symmetric monoidal. Furthermore, if $\V$ is a Markov category and $\monadM$ affine, then $\Kl_\monadM$ becomes a Markov category with copy and delete maps given by $\leftFunctor_\monadM(\copyMor_X)$ and $\leftFunctor_\monadM(\delMor_X)$.
\end{proposition}

Since many common categories such as $\Set$ are Cartesian and hence Markov categories, \cref{prop:monad-markov-category} provides a convenient way introduce uncertainty. Moreover, by the strictification theorem for Markov categories \cite[Prop.~10.16, Thm.~10.17]{fritzSynthetic2020} we will assume without loss of generality that $\Kl_\monadM$ is a strict monoidal category whenever it is convenient. We now provide some key examples.

\begin{example}
    The covariant nonempty powerset functor $\PowNoEmpty \colon \Set \to \Set$ is an affine monad with $\mu$ given by union, and $\eta$ mapping elements to singletons. It is symmetric monoidal under $\symMonadMor_{X,Y}$ which sends pairs of subsets to their Cartesian product. $\Kl_\PowNoEmpty$ consists of sets and multi-valued functions.
\end{example}

\begin{example}
    The Giry monad $\Dist \colon \Meas \to \Meas$, which sends a measurable space $(X, \SigAlg_X)$ to the space of probability measures on it (with an appropriate $\sigma$-algebra), forms an affine symmetric monoidal monad \cite[Section 4]{fritzSynthetic2020}. It's Kleisli category is $\Stoch$, the category of Markov kernels.
\end{example}

\begin{example}
    The arrow functor $\Arr \colon \Pos \to \Pos$ maps each poset to the set of intervals $\interval{a}{b}$, ordered by their end-points: $\interval{a}{b} \posetleq \interval{c}{d}$ iff $a \posetleq c \wedge b \posetleq d$. 
    It forms an affine symmetric monoidal monad with unit $\monadUnit_\posetP: a \elementTo \interval{a}{a}$, multiplication $\monadMul_\posetP \colon \interval{\interval{a}{b}}{\interval{c}{d}} \elementTo \interval{a}{d}$, and $\symMonadMorOf{\posetP}{\posetQ}: \interval{a}{b} \otimes \interval{c}{d} \elementTo \interval{a \otimes c}{b \otimes d}$.
\end{example}

\begin{example}
    The twisted arrow functor $\TwiArr: \Pos \to \Pos$ orders intervals by inclusion: $\interval{a}{b} \posetleq \interval{c}{d}$ iff $c \posetleq a \wedge b \posetleq d.$ Although it is of interest as an uncertainty semantics \cite{censiUncertainty2017}, one cannot define natural transformations $\monadMul$ and $\monadUnit$ to make it into a monad.
\end{example}
Using change-of-base, we can endow the hom-sets of an SMC with uncertainty.

\begin{theorem} \label{thm:monad-uncertainty}
    Given a $\V$-category $\C$ and a symmetric monoidal monad $(\monadM, \monadUnit, \monadMul, \symMonadMor)$ on $\V$, there is a $\V$-category $\monadM_*\C$ that has the same objects as $\C$, hom-objects given by $\monadM \C(X,Y)$, composition $\symMonadMor \then \monadM(\then_\C)$, and identities $\eta_I \then \monadM(\id_A)$. There is an identity-on-objects $\V$ functor $\iota \colon \C \to \monadM_* \C$ that maps arrows via $\eta_{\C(X,Y)} \colon \C(X,Y) \to \monadM \C(X,Y)$. If $\C$ is (symmetric) monoidal, then so is $\monadM_* \C$, and the underlying functor $\iota_0 \colon \C_0 \to (\monadM \C)_0$ is strict (symmetric) monoidal.
\end{theorem}

\begin{proof}
    Since $\monadM$ is a symmetric monoidal functor with strength $\widetilde \monadM = \symMonadMor$ and unit $\monadM_\epsilon = \monadUnit_I$, we can directly apply \cref{prop:change-of-base,prop:change-of-base-monoidal-transfer}. $\V$-functoriality of $\iota$ holds since $\monadUnit$ is a monoidal transformation.
\end{proof}

\begin{example}[Imprecise probability theory]
    Imprecise probability theories incorporate uncertainty about the choice of probabilities during statistical modeling \cite{walleyStatistical1991, halpernReasoning2005}. One general version considers subsets of probability measures. This idea is realized by $\PowNoEmpty_* \Stoch$, whose arrows are subsets of Markov kernels that compose by taking the set of possible composites. The strict symmetric monoidal inclusion $\iota_0 \colon \Stoch \to \PowNoEmpty_* \Stoch$ allows us to transfer the supply of comonoids $\copyMor$ and $\delMor$. Since $\{k : X \to Y\} \then \iota_0(\delMor_Y) = \{k \then \delMor_Y \} = \{ \delMor_X \} = \iota_0(\delMor_X)$, the result is again a Markov category. 
    Other common flavors of imprecision are captured by our constructions. For instance, applying $\Arr$ to the $\Pos$-category of sub-probability measures (ordered point-wise) yields upper and lower probabilities, while robust Bayesian methods can be modeled using the parametric uncertainty introduced in \cref{ex:external-parametrization}.
\end{example}

\subsection{Partial slice functor}\label{subsec:slice-functor}
This section defines the partial slice functor $\U \sliceFunctor \colon \W \to \CCat$, which will be used to apply parameterization to the hom-sets of a category. We prove that it is monoidal when $\U$ is strict, and demonstrate how SMC structure transfers along it.

\begin{definition}[Partial slice functor]\label{def:slice-functor}
    Given a category $\W$ and a subcategory $i \colon \U \hookrightarrow \W$, the \emph{partial slice functor} $\U \sliceFunctor \colon \W \functorArr \smcat$ sends each object $A \in \Ob(\W)$ to the comma category $\commacat{i}{\Delta_A}$ with
    \begin{compactitem}
        \item[(i)] objects given by pairs $(U \in \Ob(\U),f \colon i(U) \to A)$,
        \item[(ii)] arrows $\varphi \colon (U,f) \to (V,g)$ given by $\U$-arrows $\varphi \colon U \to V$ satisfying $f = i(\varphi) \then g$. 
    \end{compactitem}On arrows, $\U \sliceFunctor$ sends each $f \colon A \to B$ in $\W$ to the post-composition functor $\U \sliceWith{f} \colon \U\sliceWith{A} \functorArr \U\sliceWith{B}$. 
\end{definition}

\begin{proposition}
    Let $(\W, \prodW, \unitW)$ be an SMC and $i: \U \hookrightarrow \W$ a full strict monoidal subcategory. Then the partial slice functor $\U \sliceFunctor \colon \W \to \smcat$ is monoidal under the following comparison arrows in $\smcat$:
    \begin{compactitem}
        \item[(i)] The functor $\U \sliceFunctor_\epsilon \colon \unitCat \functorArr \U \sliceWith{\unitW}$ maps the single object $*$ to the identity $\idUnitW \colon \unitW \to \unitW$.
        \item[(ii)] The natural transformation $\widetilde \U \sliceFunctor_{A_1,A_2} \colon \U \sliceWith{A_1} \prodSmcat \U \sliceWith{A_2} \functorArr \U \sliceWith{(A_1 \prodW A_2)}$ maps objects $f_1 \colon U_1 \to A_1$ and $f_2 \colon U_2 \to A_2$ to $f_1 \otimes f_2 \colon U_1 \otimes U_2 \to A_1 \otimes A_2$, and arrows $\varphi \colon f_1 \to g_1$ and $\psi \colon f_2 \to g_2$ to $\varphi \otimes \psi$.
    \end{compactitem}
    Moreover, $\U \sliceFunctor$ is symmetric up to the natural invertible 2-cells $\sigma_{U_1,U_2}$.
\end{proposition}

\begin{proof}
    $\U \sliceFunctor_\epsilon$ is obviously a functor, while $\widetilde \U \sliceFunctor_{A_1,A_2}$ inherits functoriality from $\otimes$. 
    For naturality and symmetry we require
    \[
    \begin{tikzcd}[cramped, row sep = small]
    	{\U \sliceWith{A_1} \times \U \sliceWith{A_2}} & {\U \sliceWith{(A_1 \otimes A_2)}} \\
    	{\U \sliceWith{B_1} \times \U \sliceWith{B_2}} & {\U \sliceWith{(B_1 \otimes B_2)}}
    	\arrow["{\widetilde \U \sliceFunctor}", from=1-1, to=1-2]
    	\arrow["{\U \sliceWith{f_1} \times \U \sliceWith{f_2}}"', from=1-1, to=2-1]
    	\arrow["{\U \sliceWith{f_1 \otimes f_2}}", from=1-2, to=2-2]
    	\arrow["{\widetilde \U \sliceFunctor}"', from=2-1, to=2-2]
    \end{tikzcd}
    \qquad \qquad
    \begin{tikzcd}[cramped, row sep = scriptsize]
    	{\U \sliceWith{A_1} \times \U \sliceWith{A_2}} & {\U \sliceWith{(A_1 \otimes A_2)}} \\
    	{\U \sliceWith{A_2} \times \U \sliceWith{A_1}} & {\U \sliceWith{(A_2 \otimes A_1)}}
    	\arrow["{\widetilde \U \sliceFunctor}", from=1-1, to=1-2]
    	\arrow["{\sigma_\smcat}"', from=1-1, to=2-1]
    	\arrow["{\U \sliceWith{\sigma_\W}}", from=1-2, to=2-2]
    	\arrow["{\widetilde \U \sliceFunctor}"', from=2-1, to=2-2]
    \end{tikzcd}
    \]
    Consider objects $h_i \colon U_i \to A_i $ and $k_i \colon V_i \to A_i$ in $\U \sliceWith{A_i}$, along with arrows $\varphi_i \colon h_i \to k_i$. Naturality follows from functorialty of $\otimes$.
    \[
        \begin{tikzcd}[cramped, row sep = tiny, column sep = 0em]
        	{U_1 \otimes U_2} && {V_1 \otimes V_2} \\
        	& {B_1 \otimes B_2}
        	\arrow["{\varphi_1 \otimes \varphi_2}", from=1-1, to=1-3]
        	\arrow["{(h_1 \otimes h_2) \then (f_1 \otimes f_2)}"', from=1-1, to=2-2]
        	\arrow["{(k_1 \otimes k_2) \then (f_1 \otimes f_2)}", from=1-3, to=2-2]
        \end{tikzcd}
    \quad = \quad
    \begin{tikzcd}[cramped, row sep = tiny, column sep = 0em]
    	{U_1 \otimes U_2} && {V_1 \otimes V_2} \\
    	& {B_1 \otimes B_2}
    	\arrow["{\varphi_1 \otimes \varphi_2}", from=1-1, to=1-3]
    	\arrow["{(h_1 \then f_1) \otimes (h_2 \then f_2)}"', from=1-1, to=2-2]
    	\arrow["{(k_1 \then f_1) \otimes (k_2 \then f_2)}", from=1-3, to=2-2]
    \end{tikzcd}
    \]
    The associativity and unitality conditions follow since $\U$ is strict monoidal.
    The symmetry condition
    \[
    \begin{tikzcd}[cramped, row sep = tiny, column sep = 0em]
    	{U_1 \otimes U_2} && {V_1 \otimes V_2} \\
    	& {A_2 \otimes A_1}
    	\arrow["{\varphi_1 \otimes \varphi_2}", from=1-1, to=1-3]
    	\arrow["{(h_1 \otimes h_2) \then \sigma_{A_1,A_2}}"', from=1-1, to=2-2]
    	\arrow["{(k_1 \otimes k_2) \then \sigma_{A_1,A_2}}", from=1-3, to=2-2]
    \end{tikzcd}
    \quad = \quad
    \begin{tikzcd}[cramped, row sep = tiny, column sep = 0em]
    	{U_2 \otimes U_1} && {V_2 \otimes V_1} \\
    	& {A_2 \otimes A_1}
    	\arrow["{\varphi_2 \otimes \varphi_1}", from=1-1, to=1-3]
    	\arrow["{h_2 \otimes h_1}"', from=1-1, to=2-2]
    	\arrow["{k_2 \otimes k_1}", from=1-3, to=2-2]
    \end{tikzcd}
    \]
    holds up to natural 2-cells obtained by joining the two sides of the equation by $\sigma_{U_1,U_2} \colon U_1 \otimes U_2 \to U_2 \otimes U_1$ and $\sigma_{V_1,V_2} \colon V_1 \otimes V_2 \to V_2 \otimes V_1$ to form a commutative parallelogram.
\end{proof}

\begin{remark}
    If $\U$ is not strict, the associativity and unitality diagrams for $\U \sliceFunctor$ also only commute up to a natural invertible 2-cell. A weaker version of change-of-base still works in this setting (see \cite[Ch.~13]{garnerEnriched2015}), but composition is no longer strictly unital and associative, yielding a bicategory.
\end{remark}

Although the partial slice is not symmetric on the nose, one can still transfer SMC structure along it, yielding a nearly strict symmetric monoidal 2-category (\cref{def:nearly-symmetric-monoidal-2cat}).

\begin{proposition} \label{prop:slice-monoidal-transfer}
    Let $(\C, \boxtimes, J)$ be a (symmetric) monoidal $\W$-category and $i \colon \U \hookrightarrow \W$ a full strict monoidal subcategory. Then $\U \sliceFunctor_* \C$ is a nearly strict (symmetric) monoidal 2-category with tensorators given by symmetries of $\W$, and coherence 1-cells given by those of $\C$. Unlike the associator and unitors, the symmetry is only natural up to 2-cells $\sigma^\W_{U_2,U_1}$. The inclusion 2-functor $\iota \colon \C_0 \to \U \sliceFunctor_* \C$ strictly preserves monoidal products and coherence 1-cells.
\end{proposition}
\begin{proof}
    Denote $\U \sliceFunctor$ by $N$. The definition of $\boxtimes^N$ from \cref{prop:change-of-base-monoidal-transfer} still makes sense. One checks that $(\widetilde N)_*$ defines a pseudofunctor, with tensorator 2-natural isomorphism 
    \[\vartheta_{f_1,f_2,g_1,g_2} \colon (f_1 \boxtimes^N f_2) \then (g_1 \boxtimes g_2) \Rightarrow (f_1 \then g_1) \boxtimes^N (f_2 \then g_2)\]
    given by the symmetry $m$ used in \cref{def:monoidal-prod-Vcats}. 
    The coherence maps $a,r$ and $l$ lift to strict 2-natural transformations by the same argument used in \cref{prop:change-of-base-monoidal-transfer}. One checks that the twisted symmetry defines a strong 2-natural transformation. 
    By definition, the composition and monoidal product in $N_*\C$ coincide with those of $\C_0$. Hence, the lifted coherence maps satisfy the axioms they do in $\C_0$, showing that $N_*\C$ is a nearly strict symmetric monoidal 2-category.
    It also follows that $\iota$ strictly preserves the monoidal product and coherence 1-cells. 
\end{proof}

\subsection{Main construction} \label{sec:main-construction}
Our main result combines uncertainty monads with parametrization derived from the partial-slice functor. We include an adapter functor $F \colon \V \to \W$, in case the monad does not operate directly on $\V$.

\begin{theorem} \label{thm:main-construction}
    Let $(\monadM,\monadMul,\monadUnit,\symMonadMor)$ be a symmetric monoidal monad on $\W$ and let $F \colon \V \to \W$ be a symmetric monoidal functor. Given a full strict monoidal subcategory $i \colon \U \hookrightarrow \W$, we write
    \[
        \ParaU{\U}{\monadM}{F} \colon \V \overset{F}{\longrightarrow} \W \overset{\leftFunctor_\monadM}{\longrightarrow} \Kl_\monadM \overset{\U \sliceFunctor}{\longrightarrow} \smcat.
    \]
    For any $\V$-category $\C$, we obtain of a 2-category $\ParaU{\U}{\monadM}{F}_*\C$ that adds \emph{parametric $\monadM$-uncertainty} to $\C$ with parameter spaces in $\U$. If $\C$ is (symmetric) monoidal, then $\ParaU{\U}{\monadM}{F}_*\C$ is a nearly strict (symmetric) monoidal 2-category.
    Unpacking the change-of-base yields the following explicit definition:
    \begin{compactitem}
        \item[(i)] The 0-cells are $\C$-objects.
        \item[(ii)] The 1-cells between $X$ and $Y$ are $\Kl_\monadM$-arrows $f \colon U_1 \to F\C(X,Y)$, where $U_1$ is an object in $\U$,
        \item[(iii)] The 2-cells between $f \colon U_1 \to F\C(X,Y)$ and $g \colon V_1 \to F\C(X,Y)$ are $\Kl_\monadM$-arrows $\varphi \colon U_1 \to V_1$ satisfying $f = \varphi \then g$. They compose vertically by $\Kl_\monadM$-composition.
        \item[(iv)] 1- and 2-cells compose horizontally according to
        \begin{equation} \label{eq:horizontal-comp-param}
            \begin{tikzcd}[cramped, sep = scriptsize, column sep = tiny]
            	{U_1} & {V_1} \\
            	& {F\C(X,Y)}
            	\arrow["{\varphi_1}", from=1-1, to=1-2]
            	\arrow["{f_1}"', from=1-1, to=2-2]
            	\arrow["{g_1}", from=1-2, to=2-2]
            \end{tikzcd}
            \:,\:
            \begin{tikzcd}[cramped, sep = scriptsize, column sep = tiny]
            	{U_2} & {V_2} \\
            	& {F\C(Y,Z)}
            	\arrow["{\varphi_2}", from=1-1, to=1-2]
            	\arrow["{f_2}"', from=1-1, to=2-2]
            	\arrow["{g_2}", from=1-2, to=2-2]
            \end{tikzcd}
            \quad \mapsto \quad
            \begin{tikzcd}[cramped, sep = scriptsize]
            	{U_1 \otimes U_2} & {V_1 \otimes V_2} \\
            	& {F\C(X,Y) \otimes F\C(Y,Z)} & {F\C(X, Z)}
            	\arrow["{\varphi_1 \otimes \varphi_2}", from=1-1, to=1-2]
            	\arrow["{f_1 \otimes f_2}"', from=1-1, to=2-2]
            	\arrow["{g_1 \otimes g_2}", from=1-2, to=2-2]
            	\arrow["{(*)}", from=2-2, to=2-3]
            \end{tikzcd}
        \end{equation}
        where $(*) = \widetilde F \then F(\then_\C) \then \monadUnit$ denotes the lifted composition of $\C$.
        The identities are given by 1-cells $\id_X := F_\epsilon \then F(\id_X) \then \monadUnit \colon I_\W \to F\C(X,X)$.
        \item[(v)] The monoidal product of 1- and 2-cells is given by a diagram similar to \eqref{eq:horizontal-comp-param}, with $(*) = \widetilde F \then F(\boxtimes_\C) \then \monadUnit$ being the lifted monoidal product of $\C$. 
        It has $I_\C$ as unit. 
        Interchange only hold up to natural invertible 2-cells $\vartheta_{f_1,f_2,g_1,g_2} \colon (f_1 \boxtimes f_2) \then (g_1 \boxtimes g_2) \Rightarrow (f_1 \then g_1) \boxtimes (f_2 \then g_2)$ called tensorators, given by the map $m$ from \cref{def:monoidal-prod-Vcats} that rearranges the parameter spaces. 
        The coherence maps are lifted from those in $\C$. For example, the left unitor has components $l_A = F_\epsilon \then F(l^\C_{A}) \then \monadUnit \colon I_\W \to F\C(I_\C \boxtimes X, X)$.
        The symmetry is only natural up to a 2-cell that swaps the parameter spaces.
    \end{compactitem}
\end{theorem}
\begin{proof}
    Use \cref{prop:change-of-base,prop:change-of-base-monoidal-transfer} on the symmetric monoidal functor $F \then \leftFunctor_\monadM$. Then apply \cref{prop:slice-monoidal-transfer} to the (symmetric) monoidal $\Kl_\monadM$-category $(F \then \leftFunctor_\monadM)_*\C$ while assuming that $\Kl_\monadM$ is strict.
\end{proof}
Next, we show that $\ParaU{\U}{\monadM}{F}_*\C$ extends $\C$, allowing us to transfer additional structure.

\begin{proposition}\label{prop:change-of-base-inclusion}
    There is an identity-on-objects 2-functor $\iota \colon \C_0 \to \ParaU{\U}{\monadM}{F}_*\C$ that strictly preserves monoidal products and coherence maps. If $F_*^0$ and post-composition by $\monadUnit$ are faithful, then so is $\iota$.
\end{proposition}

\begin{proof}
   By \cref{prop:change-of-base-monoidal-transfer}, change-of-base induces a strict symmetric monoidal functor $(\monadM F)_*^0 : \C_0 \to ((\monadM F)_*\C)_0$. 
   \Cref{prop:slice-monoidal-transfer} yields a 2-functor $((\monadM F)_*\C)_0 \to \ParaU{\U}{\monadM}{F}$ that strictly preserves monoidal products.
   The composite maps $f \colon I_\V \to \C(X,Y)$ to $\delta(f) \colon I_\W \overset{F_\epsilon}{\longrightarrow} FI_\V \overset{Ff}{\longrightarrow} F\C(X,Y) \overset{\monadUnit}{\longrightarrow} \monadM F \C(X,Y)$, proving the claim about faithfulness. Moreover, the coherence maps are precisely of this form.
\end{proof}

\begin{remark}
    By \cref{prop:change-of-base-inclusion}, we can transfer any structure on $\C_0$ that is defined in terms of arrows satisfying equations involving composition, monoidal products, and coherence maps to $\ParaU{\U}{\monadM}{F}_*\C$. This includes compact closed structure, (commutative) (co)monoids, and the supplies defined in \cite{fongSupplying2020}, which transfer along essentially surjective, strict symmetric monoidal functors \cite[Prop.~3.24]{fongSupplying2020}.
\end{remark}

We now show how to use \Cref{thm:main-construction} to obtain useful 2-categories.

\begin{example}[External parameterization]\label{ex:external-parametrization}
    Given a symmetric monoidal $F \colon \V \to \W$ and full monoidal subcategory $\U \hookrightarrow \W$, we can parametrize a $\V$-category $\C$ with $\U$-arrows $A \to \C(X,Y)$ using $\ParaU{\U}{}{F}_* \C$, where we have suppressed the identity monad in the notation. We will similarly omit $F$ when $F = \Id$.
\end{example}

\begin{remark}
    A common way to introduce parametric maps is the $\Para$ construction \cite{fongBackpropFunctorCompositional2019, capucciTowards2022}. In the simplest formulation, 1-cells in $\Para(\C, \boxtimes, J)$ are maps $A \boxtimes X \to Y$ in $\C$, where $A$ is thought of as a parameter space. Hence, $\Para$ produces parametric maps internal to $\C$. When $\C$ is closed, these are equivalent to maps $A \to [X,Y]$, where $[X,Y]$ denotes the internal hom. 
    For our applications, the external representation is more useful.
    In particular, it decouples the parametrization issue from properties of $\C$. 
\end{remark}

\begin{example}[Possibilistic uncertainty]
    If $\C$ is an ordinary category, $\ParaU{\U}{\PowNoEmpty}{}_* \C$ yields parametrized subsets of arrows. If $\C$ is $\Pos$-category $\C$, $\ParaU{\U}{\Arr}{}_* \C$ contains parametrized intervals of arrows. 
\end{example}

\begin{example}[Probabilistic uncertainty]
    To use the distribution monad to endow a category with probabilistic uncertainty, we need to view it as $\Meas$-enriched. Crucially, our assignments must retain measurability of the composition and monoidal product maps. If $\C$ is an ordinary category, one can assign each set $\C(X,Y)$ the discrete $\sigma$-algebra $\powerset(\C(X,Y))$. This works well for finite and countable sets. For uncountable sets like $\Reals$, however, the discrete $\sigma$-algebras is often too large to be useful. How to proceed will depend on the application. Sometimes, it may be possible to directly construct a $\Meas$-enriched version of $\C$ with desirable $\sigma$-algebras by hand. In other cases, it might be possible to do so uniformly by constructing a symmetric monoidal functor $F: \V \to \Meas$. 
\end{example}

\begin{example} \label{ex:meas-dp}
    For $\DP$, it seems difficult to equip the posets $\DP(\posetP,\posetQ)$ with non-trivial $\sigma$-algebras in a way that makes composition measurable. For example, the natural choice of assigning $\DP(\posetP,\posetQ)$ the $\sigma$-algebra generated by its upper sets fails due to the presence of posets like $\Reals^2$ with uncountable width.

    We can bypass measurability issues by assigning $\DP(\posetP,\posetQ)$ the discrete $\sigma$-algebra. This type of discretization occurs in practice when distributions are represented by samples. To define distributions on the resulting space, we can push forward distributions from a more familiar space $A$ along a measurable map $f \colon A \to \DP(\posetP,\posetQ)$. Such maps may be obtained from an arbitrary function $g \colon A \to \DP(\posetP,\posetQ)$ by the following discretization procedure: Fix a partition $P := \{p_j\}_{j=1}^N$ of $A$ into finitely many measurable sets. We can approximate $g$ by $g_P \colon P \to \DP(\posetP,\posetQ)$, where $g_P(p_j) := g(x_j)$ for some $x_j \in p_j$. This is measurable for the discrete $\sigma$-algebra on $P$. Moreover, there is a measurable map $i \colon A \to P$ that assigns each point to the set containing it. Hence, the composition $f := i \then g_P \colon A \to \DP(\posetP,\posetQ)$ is a measurable approximation of $g$ whose fidelity increases as $P$ gets finer. Such $f$ can also be used to define Markov kernels $B \to \DP(\posetP,\posetQ)$ by composing with kernels $k \colon B \to A$. We are currently investigating how to view continuous distributions on $\DP(\posetP,\posetQ)$ in a way that is compatible with composition.
\end{example}

\begin{example}[Learning reaction networks]
    Decorated and structured cospans \cite{fongAlgebra2016, baezStructured2020} provide a systematic method for constructing hypergraph categories \cite{fongHypergraph2019} of open systems. The framework has been widely applied, for example to open chemical reaction networks \cite{baezCompositional2017}. These can be thought of as structured cospans $\Csp(\Petri)$ in a category $\Petri$ of Petri nets \cite[Section 6]{baezStructured2020}. For finite Petri nets, we can view the resulting hypergraph category as being $\FinSet$-enriched. We construct the symmetric monoidal 2-category $\ParaU{\U}{\Dist}{\Sigma}_* \Csp(\Petri)$, where $\Sigma \colon \FinSet \to \Meas$ sends each set to the discrete measurable space. The 1-cells of this category are Markov kernels $A \to \Csp(\Petri)(X,Y)$ that send each $a \in A$ to a discrete distribution of open reaction networks with interfaces $X,Y$. The hypergraph structure on $\Csp(\Petri)$ transfers along the inclusion. Hence $\ParaU{\U}{\Dist}{\Sigma}_* \Csp(\Petri)$ can represent hypotheses about open networks which can be updated through Bayesian inference on experimental data. Moreover, it incorporates the compositional structure of the network, enabling causal reasoning about experimental interventions.
\end{example}

\section{Application to co-design}\label{sec:learning-dps}
In this section, we demonstrate the utility of parametric uncertainty by applying it to the category $\DP$. 
A computational case study based on these ideas is presented in \cite{huangComposable2025}.
We adopt the co-design diagrams from \cite{zardiniCoDesignComplexSystems2023}, using the system in \cref{fig:dp-chassis-battery} as our running example.
\subsection{Representing parameterized and uncertain design problems}
As explained in \cref{ex:DP-pos-enriched} and \cref{ex:DP-pos-enriched-smc}, we can view the compact closed SMC $\DP$ as being $\Pos$-enriched, in which case we denote it $\PosDP$. By applying the change-of-base construction, we obtain parametrized design problems with chosen uncertainty semantics. Since change-of-base is functorial, traditional co-design concepts can be lifted to the new setting.

\begin{example}\label{ex:dp-param-det}
    The 2-category $\ParaU{\U}{\Id}{}_* \DP$ represents \emph{design problems parameterized by functions}, whose 1-cells between $\posetP$ and $\posetQ$ are functions $A \to \DP(\posetP, \posetQ)$. Similarly, $\ParaU{\U}{\Id}{}_* \PosDP$ models \emph{monotonically parameterized design problems}, where the parameter spaces are ordered and the functions monotone.
\end{example}

\begin{example}
    The 2-category $\ParaU{\U}{}{\PowNoEmpty}_* \DP$ models \emph{robust design problems} parametrized by functions.
    Its 1-cells between $\posetP$ and $\posetQ$ are functions $A \to \PowNoEmpty \DP(\posetP, \posetQ)$, that assign each parameter $a \in A$ a nonempty subset of design problems from $\posetP$ to $\posetQ$. The analogous construction $\ParaU{\U}{}{\Arr}_* \PosDP$ yields parametrized \emph{intervals} of design problems.
\end{example}

\begin{example}
    Using the distribution monad $\Dist$ and the functor $\Sigma \colon \Set \to \Meas$ that assigns each set the discrete measurable space, we obtain a 2-category $\ParaU{\U}{\Dist}{\Sigma}_* \DP$ of \emph{parametrized distributions} of design problems. Its  1-cells Markov kernels $A \to \DP(\posetP, \posetQ)$, where the $\DP(\posetP, \posetQ)$ carries the discrete $\sigma$-algebra. Such kernels can be constructed by the procedure described in \cref{ex:meas-dp}.
\end{example}
We now demonstrate the use of these categories in a concrete example.

\begin{example}[Electric vehicle]\label{ex:ev-chassis-battery}
The design problem in \cref{fig:dp-chassis-battery} models an electric vehicle with of two components: chassis and battery.
    \begin{figure}[tb]
        \centering
        \includegraphics[width=0.65\linewidth]{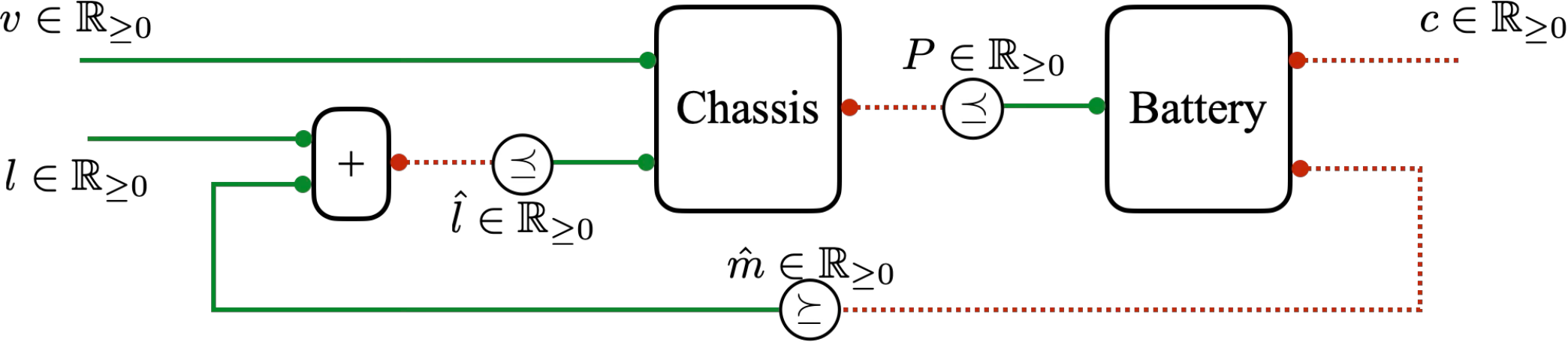}
        \caption{Design problem for an electric vehicle composed of chassis and battery components.}
        \label{fig:dp-chassis-battery}
    \end{figure}
    Blocks represent arrows in $\DP$, while wires are objects denoting resources (red) and functionalities (green).
    The battery $\BatteryDP$ provides \funsty{power} $P$ in return for \ressty{cost} $c$ and \ressty{battery mass} $\hat m$. 
    The chassis $\ChassisDP$ provides \funsty{velocity} $v$ and \funsty{total load} $\hat l$, while requiring \ressty{power} $P$.
    The functionality $l$ is the extra payload provided by the design.
    We adopt the usual order $\leq$ for $\Reals$, and write $\Reals^n$ for the $n$-fold product of posets.
    By the graphical language for \emph{compact closed SMCs}, the diagram specifies an arrow in $\DP$ representing the composite design problem.
    In practice, the design problems for the chassis and battery may be \emph{parameterized}.
    For instance, the chassis $\ChassisDP$ might be given by:
    \begin{equation*}
        \ChassisDP((v,\hat l),P;\paramTheta) = \True \iff \monfunPhi(v, \hat l;\paramTheta) \leq P,
    \end{equation*}
    where $\monfunPhi$ is a monotone function in $v$ and $\hat l$, with parameter $\paramTheta$. 
    Such dependencies can be captured in $\ParaU{\U}{\Id}{}_* \DP$. Here, $\paramTheta$ can model factors we have no control over, but need to be taken into account. Alternatively, we might introduce $\paramTheta$ to assess how \emph{sensitive} the solution of the design problem is to problem specification. Finally, we could use $\paramTheta$ to fit the design problem to empirical data.

    Uncertainty in component performance can be modeled by leveraging uncertainty monads.
    For example, suppose we are choosing from battery models $B_M$, where each model provides upper and lower bounds on performance.
    This can be formalized as a 1-cell $\BatteryDP \colon B_M \to \Arr \DP(\Reals, \Reals^2)$ in $\ParaU{\U}{\Arr}{}_*\PosDP$. 
    Similarly, we may have bounds $\ChassisDP \colon C_M \to  \Arr \DP(\Reals^2, \Reals)$ on the performance of chassis in $C_M$. The composite in $\ParaU{\U}{\Arr}{}_*\PosDP$ combines these bounds by composing the worst and best case bounds for each pair in $B_M \times C_M$. 
    If we have more detailed knowledge about component performance, we may instead use distributions to represent it. Distributions are also useful for capturing random influences on performance.
\end{example}
We can extend the graphical language for design problems to the parametrized case, as is often done for the $\Para$ construction. 
\Cref{fig:chassis-battery-parameterized} extends the electric vehicle. The 1-cells are shown as rounded boxes with incoming wires representing the parametric dependence. The sharp rectangles depict reparametrization 2-cells given by arrows in $\Kl_\monadM$. Since the 1-cells inhabit a nearly strict symmetric monoidal 2-category, the diagram has an unambiguous interpretation up to unique permutations of the parameter spaces.

\begin{figure}
    \centering
    \includegraphics[width=0.65\linewidth]{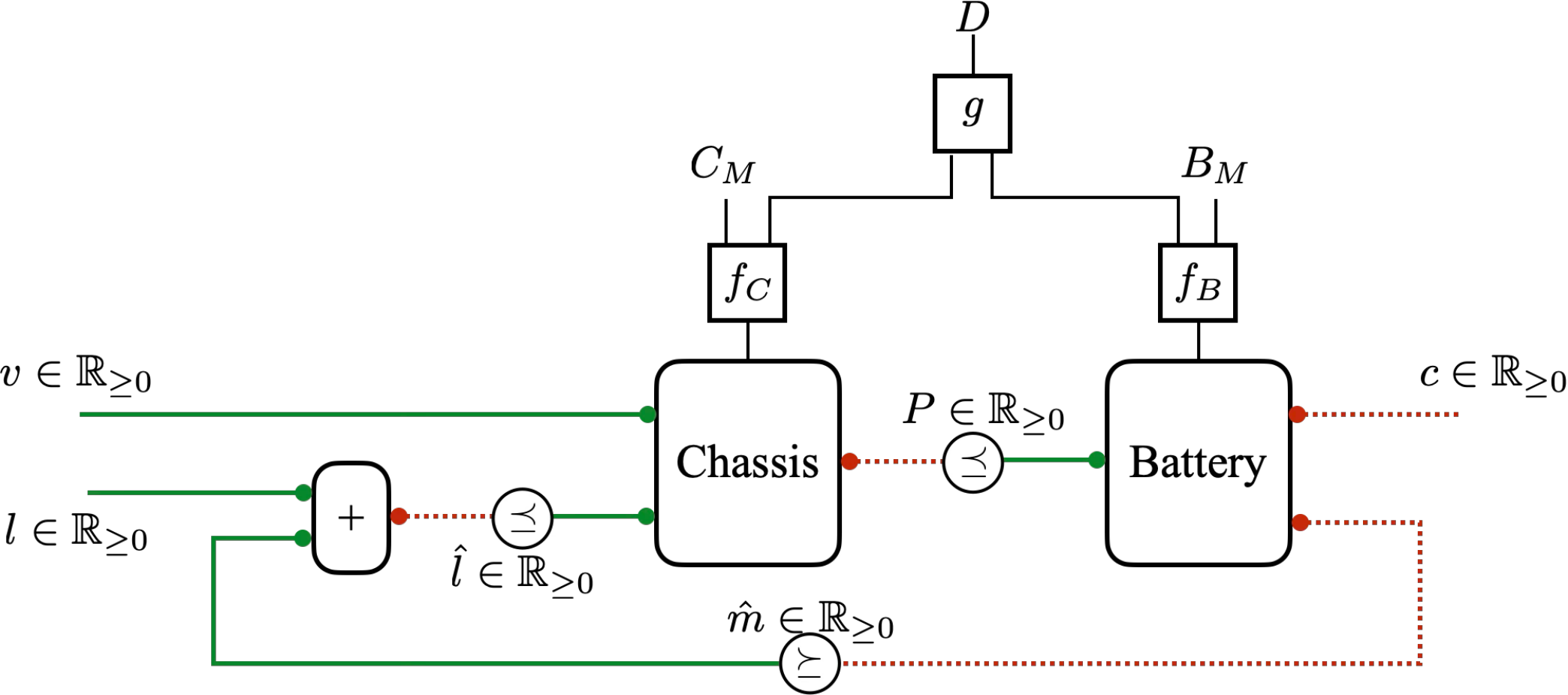}
    \caption{Parameterized version of the electric vehicle design problem. The rounded rectangles represent 1-cells $A \to \monadM F \DP(\posetP, \posetQ)$, while the rectangles represent reparameterization 2-cells given by arrows in $\Kl_\monadM$.}
    \label{fig:chassis-battery-parameterized}
\end{figure}

\begin{example}
    \Cref{fig:chassis-battery-parameterized} shows a reparameterization of the electric vehicle problem by a 2-cell composed of $f_C, f_B$, and $g$. This introduces dependencies between the components. For probabilistic uncertainty, this allows us to graphically model complex random processes that influence the design.
\end{example}

\subsection{Making decisions}

In co-design, we \emph{query} problems to inform decisions.
For instance, the query \FixFunMinRes{} maps a design problem $\designProblem \in \DP(\posetF, \posetR)$ and a functionality $f \in \posetF$ to the set of minimal resources which make $f$ feasible.
Parametrized design problems unlock additional queries, where we aim to find decision parameters that maximize some utility function on designs.

\begin{example}\label{ex:chassis-battery-decision}
    Viewing \cref{fig:chassis-battery-parameterized} as a diagram of deterministically parameterized design problems yields a map $f \colon C_M \otimes D \otimes B_M \to \DP(\PosReals^2, \PosReals)$. We think of the domain as a space of decision variables. Suppose we require our design to yield a fixed velocity and load $(v,l)$. For each possible decision, the \FixFunMinRes{} query $q$ returns the minimal cost $c$ required to satisfy our design constraints. Hence, the optimal decision is one minimizing $f \then q \colon C_M \otimes D \otimes B_M \to \PosReals$.
    In case there is probabilistic uncertainty about components, we can lift $q$ to $\Kl_\Dist$ and compose $f \then q$ as Markov kernels,  leveraging Bayesian decision theory. 
    For instance, we might minimize the expected cost, or use some risk-sensitive metric.
    The analogous method works for other types of uncertainty. For instance, in case of performance bounds, we might optimize the worst case scenario.
\end{example}

\subsection{Learning design problems}

In addition to making decisions based on known models, we often need to learn models from data. We illustrate how our construction can handle such cases by example.

\begin{example}[Optimization]
    Suppose the chassis is of the form $\monfunPhi(v,l;\paramTheta) \leq P$ for $\monfunPhi$ monotone in $(v,l)$, and that we have a dataset of observed feasible triples $y_i := (v_i,l_i,P_i)$. 
    If $y_i$ is noisy empirical data, a simple least-squares fit for $\paramTheta$ might be appropriate. However, the resulting $\hat \paramTheta$ will likely violate some of the $y_i$. Hence, if $y_i$ are guaranteed, for instance if they are the results of exact simulations, then we could choose $\hat \paramTheta \in \{\paramTheta :  \monfunPhi(v_i,l_i;\paramTheta) \leq P_i \}$ according to a metric that balances fit and complexity.
\end{example}

\begin{example}[Bayesian inference]
    Viewing \cref{fig:chassis-battery-parameterized} using probabilistic semantics returns a Markov kernel $k: C_M \otimes D \otimes B_M \to \DP(\PosReals^2, \PosReals)$. Here, $C_M$ and $B_M$ might represent design decisions, while $D$ is an unknown manufacturer-specific parameter affecting performance. Moreover, suppose we have observed a dataset of feasible triples $y_i := (v_i,l_i,c_i)$, each associated with input decisions $x_i := (c_{M,i},b_{M,i})$. Given a prior distribution $p(d) \colon I_\Meas \to D$, we can condition on the data to obtain a posterior distribution $\hat p(d \mid x_{1:n},y_{1:n}) \colon  I_\Meas \to D$. This process can be expressed diagrammatically in $\Stoch$ \cite{choDisintegration2019, fritzSynthetic2020}.
\end{example}

\begin{example}[Active learning]
     In practice, components are often implicitly defined through some expensive procedure, such as simulations or optimization problems.
     To approximate these efficiently, \emph{active learning} leverages a surrogate model.
     Starting from an initial surrogate, we select simulation settings, update the surrogate with the results, and repeat. 
     The surrogate helps identify settings likely to provide informative results.
     Any uncertainty semantics can be used, with more expressive options offering better control but higher computational cost.
     Our approach allows composing surrogates for individual components to predict information and update beliefs at the system level.
\end{example}

\bibliographystyle{eptcs}
\bibliography{generic}

\end{document}